%% file: 0main.tex
\documentclass[11pt,a4paper]{article}
\usepackage{csquotes}

\usepackage[
  backend=biber,
  style=ieee,
  dashed=false,
  sortcites=true,
  sorting=nty,      
  maxbibnames=99
]{biblatex}
\usepackage{amsmath,amsthm,amsthm,mathrsfs}
\usepackage{fancyhdr}
\usepackage[british]{babel}
\usepackage{amssymb}
\usepackage{latexsym}
\usepackage{graphicx}
\usepackage{epstopdf}
\usepackage{float}
\usepackage{color}
\usepackage[natural]{xcolor}
\usepackage{textgreek}
\usepackage{float,color,fancybox,shapepar,setspace,hyperref}
\usepackage{enumitem}
\theoremstyle{plain}
\usepackage{todonotes}
\newtheorem{theorem}{Theorem}

\newtheorem{lemma}[theorem]{Lemma}

\numberwithin{theorem}{section}

\newtheorem{definition}[theorem]{Definition}
\newtheorem{fact}[theorem]{Fact}

\normalsize 

\title{\LARGE Dense minors and bipartite independence numbers}
\author{
Xia Wang\thanks{School of Mathematics, Shandong University, Jinan, China and Extremal Combinatorics and Probability Group
(ECOPRO), Institute for Basic Science (1BS), Daejeon, South Korea. Supported by China Scholarship Council and
IBS-R029-C4. Email: {\tt xiawang@mail.sdu.edu.cn}.}
\and
Donglei Yang\thanks{School of Mathematics, Shandong University, Jinan, China. Email: {\tt dlyang@sdu.edu.cn}. Supported by Natural Science Foundation of China (12571374).}
}
\date{}

\begin{document}
\maketitle

\begin{abstract}
A graph $G$ is $m$-joined if there is an edge between every two disjoint $m$-sets of vertices.
In this paper, we prove that for any $\varepsilon>0$ and sufficiently large $m, n\in \mathbb{N}$ with $m \le n^{1-\varepsilon}$, every $n$-vertex $m$-joined graph $G$ contains a minor with density $\Omega\!\left(\tfrac{n}{\sqrt{m}}\right)$, which is best possible up to a constant factor. 
When $m \ge n^{1-\varepsilon}$, we further show that $G$ contains a clique minor of order $\Omega\!\left(\tfrac{n}{\sqrt{m\log m}}\right)$.
\end{abstract}

\input{2introduction}

\input{3preliminary}
\input{4cliqueminor}

\printbibliography

\input{8appendix}

\end{document}

%% file: 2introduction.tex
\section{Introduction}


\par
Graph minor is a central concept in graph theory and plays a crucial role in areas such as algorithm design, extremal graph theory, and the study of expanders. 
A graph $G$ is said to contain an \emph{$H$-minor} if there exist $|V(H)|$ disjoint vertex subsets $\{T_v\}_{v\in V(H)}$ of $V(G)$ such that each $T_v$ induces a connected subgraph, and for every $uv\in E(H)$ there is at least one edge joining $T_u$ and $T_v$ in $G$.

\par
Minors form a natural bridge between graph theory, topology, and geometry. 
Kuratowski's classical 
theorem~\cite{Kuratowski1930} states that a graph is planar if and only if it does not contain a topological minor of $K_5$ or $K_{3,3}$, thus establishing a deep link between planarity and minors. 
An equivalent characterization, due to Wagner~\cite{Wagner1937}, asserts that a graph is planar if and only if it has no 
$K_5$-minor or $K_{3,3}$-minor.
Later, the Robertson--Seymour theorem~\cite{ROBERTSON2004325} demonstrated that all graphs are well-quasi-ordered under the minor relation. 
In particular, any infinite family of graphs must contain two graphs such that one is a minor of the other. 
This remarkable result transforms the classification of graph families into the study of finitely many excluded minors. 

\par
In 1943, Hadwiger~\cite{hadwiger1943klassifikation} conjectured that if $\chi(G) \ge t$, then $\mathrm{ccl}(G) \ge t$, 
and he confirmed the conjecture for $t \le 4$. 
The cases $t=5,6$ were later shown to be equivalent to the Four-Color Theorem, proved respectively by Wagner~\cite{wagner1937eigenschaft} and Robertson, Seymour, and Thomas~\cite{robertson1993hadwiger}. 
However, the conjecture remains open for $t\ge7$. 
Recently, substantial progress has been made toward this conjecture: Norin, Postle, and Song~\cite{NorinPostleSong2019} proved that every $K_t$-minor-free graph is $O\big(t(\log t)^{\beta}\big)$-colorable for every fixed $\beta>1/4$; 
building on this, Delcourt and Postle~\cite{DelcourtPostle2021} showed that $K_t$-minor-free graphs are $O(t\log\log t)$-colorable and reduced the linear Hadwiger conjecture to coloring finitely many small graphs.
There has also been parallel progress on improper variants: for fixed $h$, every $K_h$-minor-free graph is $(h-1)$-colorable with bounded monochromatic components (the clustered Hadwiger conjecture)~\cite{DujmovicEsperetMorinWood2023}.
Since $\chi(G)\cdot\alpha(G)\ge |G|$, Hadwiger’s conjecture implies the following weaker statement: if $\alpha(G)\le t$, then $\mathrm{ccl}(G)\ge |G|/t$. 
Plummer, Stiebitz, and Toft~\cite{Plummer2003} showed that for $\alpha(G)=2$, this formulation is equivalent to Hadwiger’s conjecture. 
Duchet and Meyniel~\cite{Duchet1982} proved that $\mathrm{ccl}(G)\ge |G|/(2\alpha(G)-1)$, which for $\alpha(G)=2$ gives $\mathrm{ccl}(G)\ge |G|/3$. 
This bound has since been improved through a series of works~\cite{1f963a3773cd4f1e875d42ab8721d761, Kawarabayashi2005, Kawarabayashi2007, Maffray1987, Pedersen2010, Plummer2003, Woodall1987}. 
For $\alpha(G)=3$, the best known bound $\mathrm{ccl}(G)\ge |G|/4$ was obtained by Kawarabayashi and Song~\cite{Kawarabayashi2007}. 
For large $\alpha(G)$, Fox~\cite{Fox2010} showed that $\mathrm{ccl}(G)\ge |G|/((2-c)\alpha(G))$ with $c\approx1/57.5$, later improved by Balogh and Kostochka~\cite{BALOGH20112203} to $c>1/19.2$. 

\par
McDiarmid and Yolov~\cite{Mcdiarmid2017} introduced the notion of the \emph{bipartite-hole number} in the study of Hamilton cycles. 
The bipartite-hole number of $G$ is the largest integer $r$ such that $G$ contains an $(s,t)$-bipartite hole for every nonnegative $s,t$ with $s+t=r$. 
An $(s,t)$-bipartite hole consists of disjoint vertex sets $S,T$ with $|S|=s$, $|T|=t$, and $e_G(S,T)=0$. 
This notion has attracted considerable attention in recent years~\cite{Han_Hu_Ping_Wang_Wang_Yang_2024, https://doi.org/10.1002/rsa.20913, CHUDNOVSKY2020107396, doi:10.1137/21M1456509}. 
In this paper, we adopt a closely related concept—the \emph{bipartite independence number}—introduced by Nenadov and Pehova~\cite{Nenadov2020}. 

\begin{definition}[{\cite{Nenadov2020}}]
The \emph{bipartite independence number} $\alpha^*(G)$ is the largest integer $t$ such that $G$ contains a $(t,t)$-bipartite hole.
\end{definition}

Note that $\alpha^*(G) < m$ means for any disjoint sets $S,T \subseteq V(G)$ with $|S|,|T|= m$, we have $e_G(S,T) \ge 1$, which is widely known as \emph{$m$-joined}. 
Moreover, for every induced subgraph $G[X]$, it follows that $\alpha^*(G[X]) \le \alpha^*(G)$.
Immediately, we can get the following property.
\begin{fact}\label{component}
    If the graph $G$ satisfies that $\alpha^*(G) < m$, the every subset $U \subseteq V(G)$ with $|U| = 3m$ contains a connected component of size at least $m$.\footnote{
Suppose not, so each component of $G[U]$ has size at most $m-1$.  
Combine components one by one until the total size first reaches or exceeds $m$, and delete the exceeding part from the last component.  
The removed vertices are fewer than $m-1$, so the obtained $m$-set $A$ has no edges to the remaining vertices in $U$.  
Since at least $2m$ vertices remain, the same process gives another disjoint $m$-set $B$ with no edges to $A$, contradicting the assumption. }
\end{fact}
For any $n$-vertex graph $G$ with $\alpha^*(G) < m$, by Lemma \ref{component}, we can conclude that $G$ contains at least $\frac{n}{3m}$ vertex-disjoint connected subgraphs each of order at least $m$. Then there is a clique minor of order at least $n/(3m)$, as the condition $\alpha^*(G) < m$ guarantees that there is at least one edge between any two disjoint $m$-sets.
However, the actual bound could be substantially larger. \medskip

We suspect that for every $n$-vertex graph $G$ with $\alpha^*(G) < m$, one can get that
\[
  \mathrm{ccl}(G)=\Omega\!\left(\tfrac{n}{\sqrt{m}}\right).
\]
For example, consider the Erd\H os--R\'enyi random graph $G\sim G(n,p)$ with any fixed constant $p\in(0,1)$ (e.g., $p=\tfrac12$). 
A union bound over all ordered pairs $(S,T)$ of disjoint $m$-vertex subsets gives
that with high probability $\alpha^*(G)=\Theta\big(\log n\big)$.
Meanwhile, $G$ has an average degree $(1\pm o(1))pn$, and therefore, by the celebrated results of Kostochka~\cite{kostochka1984lower} and Thomason~\cite{thomason2001extremal}, or the exact estimate of clique minors in random graphs due to for example Bollob\'as, Catlin, and Erd\H os~\cite{bollobas1980complete}, and Fountoulakis, K\"uhn, and Osthus~\cite{fountoulakis2010complete}, $G$ contains a clique minor of order 
\[
  \Theta\!\left(\frac{pn}{\sqrt{\log n}}\right)
  =\Theta\!\left(\frac{n}{\sqrt{\log n}}\right)
  =\Theta\!\left(\frac{n}{\sqrt{\alpha^*(G)}}\right).
\] 
Therefore, in the regime $m=\Theta(\log n)$, if true, the expected bound $\Theta\!\left(\frac{n}{\sqrt{m}}\right)$ is tight up to constant factors. We establish the following two results to support this.

\par


\begin{theorem}
\label{thm:dense_minor}
There exists a constant $m_0>0$ such that the following holds for every $m_0\le m\le \frac{n}{12}$. If $G$ is an $n$-vertex graph with $\alpha^*(G)<m$, then $G$ contains a minor with density at least 
\[ 
  \frac{n}{96\sqrt{m\log m}}.
\]
Moreover, when $m_0\le m\le n^{1-\varepsilon}$ for any $\varepsilon>0$, $G$ contains a denser minor with density at least 
\[
  \frac{1}{96}\,\sqrt{\frac{\varepsilon}{1-\varepsilon}}\frac{n}{\sqrt{m}}.
\]
\end{theorem}

\par
When $m\le n^{1-\varepsilon}$, Theorem~\ref{thm:dense_minor} guarantees a minor of density $\Omega\!\left(\tfrac{n}{\sqrt{m}}\right)$, which is best possible as discussed above.
For $m\ge n^{1-\varepsilon}$, the theorem provides a minor of density $\Omega\!\left(\tfrac{n}{\sqrt{m\log m}}\right)$. 
Indeed, in this range, we can ensure the existence of a clique minor of comparable density, as shown below. 

\begin{theorem}
\label{thm:clique_minor}
For all $\log n^2 \le m \le n/100$, if $G$ is an $n$-vertex graph with $\alpha^*(G)<m$, then 
\[
  \mathrm{ccl}(G) \ge \frac{n}{10^5\sqrt{m\log n}}.
\]
\end{theorem}

\par
The remainder of this paper is organized as follows. 
Section~\ref{preliminary} introduces the necessary notions and tools concerning expander graphs, highlighting Lemma~\ref{usefullem}, which will be repeatedly used. 
Section~\ref{denseminor} provides the proofs of Theorem~\ref{thm:dense_minor} and Theorem~\ref{thm:clique_minor}.

%% file: 3preliminary.tex
\section{Preliminary}\label{preliminary}

\subsection{Notation}\label{notation}
\par
Let $G$ be a graph with vertex set $V(G)$ and edge set $E(G)$. 
For a subset $X\subseteq V(G)$, $G[X]$ denotes the \emph{induced subgraph} of $G$ on $X$.
For any $X,Y\subseteq V(G)$, let $G[X,Y]$ be the induced bipartite subgraph of $G$ with parts $X$ and $Y$, and let $e_G(X,Y)$ denote the number of edges in $G[X,Y]$.
Let $G-X$ denote the induced subgraph $G[V(G)\setminus X]$.
The \emph{external neighborhood} of $X$ is 
\[
N_G(X):=\{u\in V(G)\setminus X:\, uv\in E(G)\ \text{for some }v\in X\}.
\]
The \emph{length} of a path (or cycle) is the number of its edges. 
The \emph{distance} between two vertices $x,y\in V(G)$, denoted by $\mathrm{dist}_G(x,y)$, is the length of a shortest $xy$-path in $G$.
The \emph{distance} between two vertex sets $X,Y\subseteq V(G)$, denoted by $\mathrm{dist}_G(X,Y)$, is $\min\{\mathrm{dist}_G(x,y):\,x\in X,\,y\in Y\}$.
The \emph{diameter} of $G$ is $diam(G):=\max\{\mathrm{dist}_G(x,y):\,x,y\in V(G)\}$.
Let $[n]:=\{1,2,\dots ,n\}$. When it is not essential, we omit floors and ceilings, and all logarithms are natural throughout.

\subsection{\texorpdfstring{$(\alpha,t)$-expanders}{(alpha,t)-expanders}}\label{tools1}

Expanders are graphs with strong vertex-expansion properties, with applications in communication networks, error-correcting codes, and algorithmic derandomization.
In this paper, we use the following notion of expander introduced by Krivelevich and Sudakov~\cite{krivelevich2009minors}.

\begin{definition}[{\cite{krivelevich2009minors}}]
We say that a graph $G$ on $n$ vertices is an \emph{$(\alpha,t)$-expander}, for some $t\ge 1$ and $0<\alpha<1$, if for every $X\subseteq V(G)$ with $|X|\le \alpha n/t$ we have
\[
|N_G(X)|\ \ge\ t\,|X|.
\]
\end{definition}

Under the assumption regarding the bipartite independence number, we can find a large expanding subgraph with additional properties.

\begin{lemma}\label{usefullem}
Let $G$ be a graph with $n$ vertices. If $\alpha^*(G)<m$, then there exists an induced subgraph $H$ satisfying
\begin{itemize}
  \item $H$ is an $(\alpha,t)$-expander,
  \item $|H|>n-m$,
  \item $\alpha^*(H)<m$,
  \item $diam(H)\le 
  \begin{cases}
  3, & t>m,\\[2pt]
  \dfrac{3\log m}{\log(1+t)}, & t\le m,
  \end{cases}$
\end{itemize}
where $\alpha=\dfrac{n-3m}{2n}$ and $t=\dfrac{n-3m}{2m}$.
\end{lemma}

\begin{proof}
By definition, if $G$ is a $(2\alpha,t)$-expander, then it is also an $(\alpha,t)$-expander. 
Suppose $G$ is not a $(2\alpha,t)$-expander. Then there exists $X\subseteq V(G)$ with $|X|\le 2m$ such that $|N_G(X)|<t|X|\le n-3m$.
Let $X_0$ be a maximal subset with this property.
If $m\le |X_0|\le 2m$, then by the definition of the bipartite independence number we have $|N_G(X_0)|>n-|X_0|-m\ge n-3m$, a contradiction. 
Hence $|X_0|<m$. Remove $X_0$ to obtain $H:=G-X_0$. Then $|H|>n-m$ and $\alpha^*(H)\le \alpha^*(G)$, since $H$ is an induced subgraph of $G$.
It remains to verify the expansion and diameter conditions. 

First, we prove that $H$ is an $(\alpha,t)$-expander.
Suppose for contradiction that $H$ is not an $(\alpha,t)$-expander. 
Then there exists $Y_0\subseteq V(H)$ with $|Y_0|\le m$ and $|N_H(Y_0)|<t|Y_0|$. 
Consequently, $|X_0\cup Y_0|\le 2m$ and
\[
|N_G(X_0\cup Y_0)|\le |N_G(X_0)|+|N_H(Y_0)|<t|X_0|+t|Y_0|=t\,|X_0\cup Y_0|,
\]
contradicting the maximality of $X_0$. Thus $H$ is an $(\alpha,t)$-expander.

As to the diameter of $H$, for any fixed $v\in V(H)$ and $\ell\in\mathbb{N}$, by expansion, the ball of radius $\ell$ around $v$ contains at least $\min\{m,(1+t)^\ell\}$ vertices. 
If $t\ge m$, then for any distinct $u,v\in V(H)$ we have $|N_H(u)|\ge m$ and $|N_H(v)|\ge m$. 
Since $\alpha^*(H)<m$, there is an edge between $N_H(u)$ and $N_H(v)$, and hence $\mathrm{dist}_H(u,v)\le 3$, i.e., $\mathrm{diam}(H)\le 3$. 

If $t<m$, take $\ell=\frac{\log m}{\log(1+t)}$, so at least $m$ vertices lie within distance $\ell$ from $v$ in $H$. 
As $\alpha^*(H)<m$, any two vertex sets of size at least $m$ in $H$ send at least one edge between them. 
Therefore $\mathrm{diam}(H)\le \frac{2\log m}{\log(1+t)}+1\le \frac{3\log m}{\log(1+t)}$.
\end{proof}



As shown in Lemma \ref{usefullem}, the induced subgraph $H$ is not only an $(\alpha, t)$-expander, but also has a small diameter, Which is certainly more efficient for constructing a large minor. To highlight this, we will use the following result to find small connected subsets in the proof of \autoref{thm:clique_minor}, and here we employ a strategy recently introduced by krivelevich and Nenadov in \cite{krivelevich2025minorssmallsetexpanders}. A short proof will be presented in Appendix \ref{apponnectset}.

\begin{lemma}\label{connectset}
For every $0<\alpha<1$, there exists $s_0>1$ such that the following holds. 
Let $G$ be an $(\alpha,t)$-expander on $n$ vertices with $\alpha^*(G)<m$, where $t,m\ge2$. 
Let $U_1,\dots,U_q\subseteq V(G)$ be disjoint subsets of order $|U_i|\ge s$, for some $1\le q<n$ satisfying
\[
s_0\log\!\Big(\frac{\ell}{n}\Big)\le s\le \frac{n}{\log\log m}, \text{ where } \ell=\max\{2n,qs\}.
\]
Then there exists a set $T\subseteq V(G)$ of size at most
\[
|T|\le 
\begin{cases}
\dfrac{50}{\alpha}\cdot\dfrac{n}{s}\cdot\log\!\Big(\dfrac{\ell}{n}\Big), & t>m,\\[8pt]
\dfrac{50}{\alpha}\cdot\dfrac{n}{s}\cdot\log\!\Big(\dfrac{\ell}{n}\Big)\cdot\dfrac{\log m}{\log(1+t)}, & t\le m,
\end{cases}
\]
such that $G[T]$ is connected and $T\cap U_i\neq\emptyset$ for all $i\in[q]$.
\end{lemma}



%% file: 4cliqueminor.tex


\section{Dense minor}\label{denseminor}

In this section, we solely use Lemma~\ref{usefullem} to prove \autoref{thm:dense_minor}. Our construction of a dense minor is straightforward by iteratively taking a large number of vertex-disjoint small-sized connected subsets, whist keeping many connections between them.

\begin{proof}[Proof of \autoref{thm:dense_minor}]
For the first part we will set $K=16$ and   
 $k:=\frac{ n}{3\sqrt{K m\log m}}$. Let $\mathcal{P}$ be an ordered partitions $V(G)=B_1\cup B_2\cup\cdots\cup B_q\cup C$ for a maximal $q\ge 0$ such that:
\begin{itemize}
  \item[(P1)] $|B_i|=p:=\sqrt{K m\log m}$ and $G[B_i]$ is connected for every $i\in[q]$;
  \item[(P2)] for every $i\in[\lceil \frac{m}{p}\rceil,\,q-1]$, the block $B_{i+1}$ has neighbors in at least $\tfrac12\,(i-\frac{m}{p})$ distinct blocks among $\{B_1,\dots,B_i\}$. 
\end{itemize}

Since $m\le n/12$, we have $k\ge \frac{4m}{p}$. 
If $q\geq k$, then we contract each $B_i$, $i\in[q]$, to a single vertex and keep all edges between distinct blocks. By (P2), the resulting minor on $q$ vertices has average degree at least
\[
\frac{1}{q}\sum_{i=\lceil m/p\rceil+1}^{q}\frac{1}{2}\Big(i-\frac{m}{p}\Big)
\ \ge\ \frac{q}{4}-\frac{m}{2p}
\ \ge\ \frac{q}{8}
\ \ge\ \frac{n}{24\sqrt{K m\log m}}=\frac{n}{96\sqrt{m\log m}}.
\]

Suppose for contradiction that $q<k$. Then $|B|:=|\bigcup_{i=1}^q B_i|< kp\leq  n/3$, and $|C|\ge 2n/3$.
By Lemma~\ref{usefullem}, the induced subgraph $G[C]$ contains a subgraph $G[D]$ such that:
    \begin{itemize}
        \item $G[D]$ is an $(\alpha, t)$-expander where $\alpha=\frac{|C|-3m}{2|C|}$ and $t=\frac{|C|-3m}{2m}$,
        \item $|D|>|C|-m\geq \frac{n}{2}$,
        \item $\alpha^*(G[D])<m$,
        \item $diam(G[D])\leq \max \{3, \frac{3\log m}{\log(1+t)}\}$.
    \end{itemize}
    
Arbitrarily choose $m/p$ blocks among $\{B_1,\dots,B_q\}$ and let $B'$ be the union of these blocks. Then $|B'|=m$ and
\[
|N_D(B')|
\ \ge\ |N_G(B')|-|B|-m
\ \ge\ n-m-\frac{n}{3}-m
\ \ge\ \frac{n}{2},
\]
using $m\le n/12$. By averaging, there exists a block $B_i$ inside $B'$ satisfying
\[
|N_D(B_i)|\ \ge\ \frac{|N_D(B')|}{m/p}\ \ge\ \frac{n\sqrt{K\log m}}{2\sqrt{m}}=:u,
\]
and we call such a block \emph{good}. The number of good blocks is at least $q-m/p$ as we can find a good one among any collection of $m/p$ blocks.

Pick a random set $W\subseteq D$ by choosing $\frac{|D|}{u}$ vertices uniformly and independently (with repetitions). 
For a fixed good $B_i$, we have
\[
\Pr\big[W\cap N_D(B_i)=\varnothing\big]\ \le\ \Big(1-\frac{|N_D(B_i)|}{|D|}\Big)^{\frac{|D|}{u}}\ \le\ \Big(1-\frac{u}{|D|}\Big)^{\frac{|D|}{u}}\le\ \frac{1}{e}.
\]
Let $X$ be the random variable counting all good blocks $B_i$ with $W\cap N_D(B_i)\ne\varnothing$. Then
\[
\mathbb{E}[X]\ \ge\ (1-\frac{1}{e})\Big(q-\frac{m}{p}\Big)\ \ge\ \frac12\Big(q-\frac{m}{p}\Big).
\]
Hence there exists a choice of $W$ that hits at least $\tfrac12\,(q-m/p)$ good blocks. Fix a vertex $w_0\in W$, and let $Z$ be the union of all shortest paths in $G[D]$ connecting $w_0$ to each vertex of $W$. 
The total number of vertices in $Z$ satisfies
\[
|V(Z)| \le |W|\cdot diam(G[D]) \leq \frac{n}{u}\cdot (3+\frac{3 \log m}{\log(1+t)})\leq \sqrt{\frac{4m}{K\log m}}\cdot 6\log m \leq p,\] 
in which the the second last inequality makes use of $|C|\geq \frac{2n}{3}$ and $m\leq \frac{n}{12}$, while the last inequality holds as $K\geq 12$.
By adding a few extra vertices if necessary, we can enlarge $Z$ to a connected subset $B_{q+1}\subseteq D$ of size exactly $p$. 
By construction, $B_{q+1}$ is adjacent to at least $\tfrac{1}{2}\big(q-m/p\big)$ of the existing blocks $B_i$ with $1\le i\le q$. 
This contradicts the maximality of $q$, and hence completes the first part of the proof (taking $K=16$ yields the desired estimate).

\medskip
For the second part, we instead set
\[
K:=\frac{12(1-\varepsilon)}{\varepsilon},\qquad k:=\frac{n}{3\sqrt{Km}},\qquad p:=\sqrt{Km},
\]
and similarly define $\mathcal{P}=\{B_1,B_2,\cdots,B_q, C\}$ as before, except that now we have $|B_i|=p=\sqrt{Km}$ (and (P2) remains unchanged).
If $q\ge k$, then the same calculation as above yields a minor of density at least
\[
\frac{1}{q}\sum_{i=\lceil m/p\rceil+1}^{q}\frac{1}{2}\Big(i-\frac{m}{p}\Big)
\ \ge\ \frac{q}{8}
\ \ge\ \frac{n}{24\sqrt{Km}}
\ =\ \frac{1}{96}\sqrt{\frac{\varepsilon}{1-\varepsilon}}\cdot\frac{n}{\sqrt{m}}.
\]
Otherwise $q<k$ implies $|B|<kp=n/3$ and hence $|C|\ge 2n/3$. Similarly, applying Lemma~\ref{usefullem} to $G[C]$ yields an induced subgraph $G[D]$ with $|D|>n/2$, $\alpha^*(G[D])<m$, and
\[diam(G[D])\ \le\ \max\!\Big\{3,\,\frac{3\log m}{\log(1+t)}\Big\},\qquad t=\frac{|C|-3m}{2m}.
\]
Since $m\le n^{1-\varepsilon}$ and $|C|\ge 2n/3$, we have
$t\ \ge\ \frac{\frac{2n}{3}-3m}{2m}\ \ge\ \frac{1}{6}\,n^{\varepsilon},$ which implies that $\frac{3\log m}{\log(1+t)}\ \le\ \frac{3(1-\varepsilon)\log n}{\log\big(1+\frac{1}{6}n^{\varepsilon}\big)}\ \le\ \frac{6(1-\varepsilon)}{\varepsilon}$, 
for all sufficiently large $n$. 
Hence $diam(G[D])\le \frac{6(1-\varepsilon)}{\varepsilon}$.\medskip

Define a block $B_i$ to be \emph{good} if
\[
|N_D(B_i)|\ \ge\ u':=\frac{n\sqrt{K}}{2\sqrt{m}}.
\]
The same averaging and random-sampling argument as above shows that there exists a set $W\subseteq D$ of size $|W|=|D|/u'$ that intersects at least $\tfrac12\,(q-m/p)$ good blocks. Taking shortest paths from a fixed $w_0\in W$ to all vertices in $W$ inside $G[D]$ yields a connected set of size at most
\[
|W|\cdot \mathrm{diam}(G[D])\ \le\ \frac{|D|}{u'}\cdot \frac{6(1-\varepsilon)}{\varepsilon}
\ \le\ \frac{n}{u'}\cdot \frac{6(1-\varepsilon)}{\varepsilon}
\ =\ \frac{2\sqrt{m}}{\sqrt{K}}\cdot \frac{6(1-\varepsilon)}{\varepsilon}
\ =\ \sqrt{Km}
\ =\ p.
\]
 After adding additional vertices if necessary we obtain a block $B_{q+1}$ of size exactly $p$ that meets at least $\tfrac12\,(q-m/p)$ previous blocks, again contradicting the maximality of $q$. This completes the proof.
    \end{proof}

\subsection{Clique minor}\label{cliqueminor}

In this section we shall use Lemma~\ref{connectset} for the proof of \autoref{thm:clique_minor}.

\begin{proof}[Proof of \autoref{thm:clique_minor}]

Let $H$ be the subgraph of $G$ obtained from Lemma~\ref{usefullem}, satisfying
\begin{itemize}
        \item  $H$ is an $(\alpha, t)$-expander,
        \item  $h:=|H|>n-m$,
        \item $\alpha^*(H)<m$,
       
    \end{itemize}
where $\alpha=\frac{n-3m}{2n}\ge \frac{1}{3}$ and $t=\frac{n-3m}{2m}$. 
We will show that $H$ contains a clique minor of order $\Omega\!\left(\frac{n}{\sqrt{m\log n}}\right)$.

\medskip
Let $\mathbf{P}$ be the family of all possible ordered partitions 
$V(H)=W\cup B_1\cup B_2\cup\dots\cup B_{q'}\cup R$ for some integer $q'\ge0$
satisfying the following:
\begin{itemize}
  \item $|B_i|=p:=60\sqrt{\frac{m\log n}{\alpha}}$ and $H[B_i]$ is connected for all $i\in[q']$;
  \item there is at least one edge between any distinct $B_i,B_j$ for any distinct $i,j\in [q']$;
  \item $|W|\le\frac{\alpha h}{2t}$ and $|N_R(W)|<\frac{t|W|}{2}$.
\end{itemize}
Clearly $\mathbf{P}\ne\emptyset$ by taking $W=\emptyset$, $q'=0$, $R=V(H)$. Choose a partition $\mathcal{P}\in \mathbf{P}$ maximizing $|W|$, and among those, one that maximizes $q'$.  
Set
\[
k:=\frac{\alpha h}{8p}\ge \frac{n}{10^5\sqrt{m\log n}}.
\]
If $q'\ge k$, then the subgraph induced by $\{B_1,\dots,B_{q'}\}$ already yields the desired clique minor and we are done.
Suppose for a contrary that $q'<k$. 
Let $B=\bigcup_{i=1}^{q'}B_i$. Then $|B|<\frac{\alpha h}{8}$ and so $|R|\ge \frac{h}{2}$.\medskip 

We first claim that $R_i:=N_R(B_i)$ satisfies $|R_i|\geq \frac{t|B_i|}{2}$ for every $i\in [q']$.
Suppose for a contradiction that some $B_j$ has $|R_j|<\tfrac{t|B_j|}{2}$. Then $|N_R(B_j\cup W)|<\tfrac{t|B_j\cup W|}{2}$.  
By the maximality of $|W|$, we must have $|B_j\cup W|>\tfrac{\alpha h}{2t}$. 
Since $|B_j\cup W|\le |B_j|+|W|\le p+\tfrac{\alpha h}{2t}\le \tfrac{\alpha h}{t}$ (for large $m$), 
and $H$ is an $(\alpha,t)$-expander, we get that
\[
|N_H(B_j\cup W)|\ge t|B_j\cup W|.
\]
 Then \[|N_R(B_j\cup W)|\geq t|B_j\cup W|-|B|\geq \frac{t|B_j\cup W|}{2}.\]
 This contradicts the assumption. 

We further claim that $H[R]$ is an $(\frac{\alpha}{2},\frac{t}{2})$-expander.
Suppose not and then by definition there exists $X\subseteq R$ with $|X|\le \tfrac{\alpha|R|}{t}$ such that $|N_R(X)|<\tfrac{t|X|}{2}$.  
This implies $|N_R(X\cup W)|\le \tfrac{t|X\cup W|}{2}$.  
By the maximality of $W$, we must have $|X\cup W|>\tfrac{\alpha h}{2t}$.  
If $\tfrac{\alpha h}{2t}<|X\cup W|\le\tfrac{\alpha h}{t}$, 
then, since $H$ is an $(\alpha,t)$-expander,
\[
|N_R(X\cup W)|\ge t|X\cup W|-|W|\ge \tfrac{t|X\cup W|}{2},
\]
a contradiction.  
Therefore $|X\cup W|>\tfrac{\alpha h}{t}$.  
Choose $X'\subseteq X\cup W$ with $|X'|=\tfrac{\alpha h}{t}$. 
Then
\begin{align*}
|N_R(X\cup W)|
  &\ge |N_H(X')| - |(X\cup W)\setminus X'| - |W|\\
  &\ge \alpha h - \frac{\alpha|R|}{t} - \frac{\alpha h}{2t} + \frac{\alpha h}{t} - \frac{\alpha h}{2t}
  \ \ge\ \tfrac{t|X\cup W|}{2},
\end{align*}
again a contradiction. Thus $H[R]$ is indeed an $(\tfrac{\alpha}{2},\tfrac{t}{2})$-expander.

Finally, we claim that the diameter of $H[R]$ satisfies
\[
\mathrm{diam}(H[R])\le
\ell:=
\begin{cases}
3,& t/2>m,\\[4pt]
\dfrac{3\log m}{\log(1+t/2)},& t/2\le m.
\end{cases}
\]
In fact, as $H[R]$ is an induced subgraph of $H$, we have $\alpha^*(H[R])<m$.  
Combining this with the fact that $H[R]$ is an $(\tfrac{\alpha}{2},\tfrac{t}{2})$-expander and applying the same argument as in Lemma~\ref{usefullem} yields the stated bound. \medskip

Now we shall apply Lemma~\ref{connectset} to $H[R]$ with $U_i=R_i$, for $i\in [q']$.
By the first claim, we know that $|R_i|\geq tp/2$. Apply Lemma~\ref{connectset} with sets $U_1,\dots,U_{q'}$, which we can do as $h\ge tp\log\log m/2$ for sufficiently large $m$, satisfying the conditions of Lemma~\ref{connectset}.  
Thus there exists a connected set $T\subseteq R$ of size at most
\[
|T|\le
\begin{cases}
\frac{50}{\alpha/2}\cdot\frac{n}{tp/2}\cdot
\log\!\Big(\frac{(\frac{\alpha h}{8p})(\frac{tp}{2})}{n/2}\Big)\le p, & t/2>m,\\[8pt]
\frac{50}{\alpha/2}\cdot\frac{n}{tp/2}\cdot
\log\!\Big(\frac{(\frac{\alpha h}{8p})(\frac{tp}{2})}{n/2}\Big)
\cdot\frac{\log m}{\log(1+t/2)}\le p, & t/2\le m.
\end{cases}
\]
Since $H[R]$ is connected, we can enlarge $T$ by adding extra vertices from $R$ to obtain a connected set $T'\subseteq R$ of size exactly $s$.  
Let $B_{q'+1}=T'$.  
Then $B_{q'+1}$ meets every $R_i$ and, by construction, forms edges with all existing $B_i$, contradicting the maximality of $q'$.  
Hence $q'\ge k$, and $H$ contains a clique minor of order $\Omega\!\left(\frac{n}{\sqrt{m\log n}}\right)$.
\end{proof}

%% file: 8appendix.tex
\begin{appendix}

\section{Proof of Lemma \ref{connectset}}\label{apponnectset}
\begin{proof}[Proof of Lemma \ref{connectset}]
Let $\ell=\max\{2n, qs\}$ and $P\subseteq V(G)$ be obtained by including each vertex independently with probability
\[
p=\frac{4}{\alpha s}\log\!\Big(\frac{\ell}{n}\Big).
\]
Since $s\ge s_0\log\!\big(\frac{\ell}{n}\big)$ and $s_0$ is sufficiently large, $p<1$. 

\smallskip
For each $i\in[q]$, let $X_i=\mathrm{dist}_G(U_i,P)$. 
Note that $\Pr[X_i=z]=\Pr[X_i\ge z]-\Pr[X_i\ge z+1]$ for $z\ge1$. 
If $t\le m$, then by Lemma~\ref{usefullem} we have $\mathrm{diam}(G)\le \frac{3\log m}{\log(1+t)}=:k$, 
hence $\Pr[X_i\ge k+1]=0$. Therefore,
\[
\mathbb{E}[X_i]=\sum_{z=1}^{k}z\Pr[X_i=z]
=\sum_{z=1}^{k}\Pr[X_i\ge z].
\]

\smallskip
Since vertices are sampled independently,
\[
\Pr[X_i\ge z]=(1-p)^{|B_G(U_i,z-1)|}\le e^{-p|B_G(U_i,z-1)|}
=\Big(\frac{n}{\ell}\Big)^{4\alpha^{-1}|B_G(U_i,z-1)|/s},
\]
where $B_G(U_i,z-1)$ denotes the ball of radius $z-1$ around $U_i$ in $G$.
As $G$ is an $(\alpha,t)$-expander, for all $z\ge1$ we have
\[
|B_G(U_i,z-1)|\ge \min\{\alpha n,\ s(1+t)^{z-1}\}.
\]
As $n\ge s\log\log m$, it follows that
\[
\mathbb{E}[X_i]
=\sum_{z=1}^{k}\Pr[X_i\ge z]
\le
\sum_{z=1}^{k}\Big(\frac{n}{\ell}\Big)^{4n/s}
+\sum_{z=1}^{\infty}\Big(\frac{n}{\ell}\Big)^{4(1+t)^{z-1}}
\le \frac{n}{\ell}.
\]

\smallskip
Let $T'\subseteq V(G)$ be the union of one shortest path from each $U_i$ to $P$. 
Then $|T'\setminus P|\le \sum_{i=1}^{q}X_i$, and hence
\[
\mathbb{E}[\,|T'\setminus P|\,]
\le \sum_{i=1}^{q}\mathbb{E}[X_i]
\le \frac{n}{s}.
\]
Since $\mathbb{E}[|P|]=pn$, by Markov’s inequality there exists a choice of $P$ satisfying simultaneously
\[
|P|\le 2pn=\frac{8n}{\alpha s}\log\!\Big(\frac{\ell}{n}\Big)
\quad\text{and}\quad
|T'\setminus P|\le \frac{2n}{s}.
\]

Fix such a vertex set $P$. Choose an arbitrary vertex $v\in P$, and for every $w\in P\setminus\{v\}$ let $L_w$ denote the set of vertices on a shortest $vw$-path in $G$. 
Let
\[
T := T'\ \cup\ \bigcup_{w\in P\setminus\{v\}} L_w.
\]
If $t\le m$, then $\mathrm{diam}(G)\le k=\frac{3\log m}{\log(1+t)}$, and hence the cardinality of $T$ is at most 
\[|T|\le \mathrm{diam}(G)\,|P|+|T'\setminus P|
   \le \frac{3\log m}{\log(1+t)}\,|P|+\frac{2n}{s}
   \le \frac{50}{\alpha}\frac{n}{s}\log\!\Big(\frac{\ell}{n}\Big)\frac{\log m}{\log(1+t)}.\]
If instead $t>m$, then by Lemma~\ref{usefullem} we have $\mathrm{diam}(G)\le3$, and a similar calculation gives
\[
|T|\le \frac{50}{\alpha}\frac{n}{s}\log\!\Big(\frac{\ell}{n}\Big).
\]
By construction, $G[T]$ is connected, $T$ intersects with every $U_i$, and $|T|$ satisfies the desired bound.
\end{proof}
\end{appendix}